\documentclass[11pt,reqno]{amsart}
\usepackage{amsmath, amssymb, amsthm}
\usepackage{url}
\usepackage[breaklinks]{hyperref}
\renewcommand{\o}{{\omega}}

\renewcommand{\(}{\left\(}
\renewcommand{\)}{\right\)}
\renewcommand{\[}{\left\[}
\renewcommand{\]}{\right\]}
\renewcommand{\i}{\infty}
\numberwithin{equation}{section}
 \theoremstyle{plain}
\newtheorem{theorem}{Theorem}[section]
\newtheorem{lemma}[theorem]{Lemma}
\newtheorem{remark}[]{Remark}
\newtheorem{definition}[]{Definition}

\newtheorem{conjecture}[theorem]{Conjecture}

\newtheorem{corollary}[theorem]{Corollary}
\newtheorem{proposition}[theorem]{Proposition}

   \makeatletter
\def\proof{\@ifnextchar[{\@oproof}{\@nproof}}
\def\@oproof[#1][#2]{\trivlist\item[\hskip\labelsep\textit{#2 Proof of\
#1.}~]\ignorespaces}
\def\@nproof{\trivlist\item[\hskip\labelsep\textit{Proof.}~]\ignorespaces}

\makeatother

\newenvironment{proof-alt}
   {\vskip 0.15in \par\noindent{\it Proof of Proposition \ref{HSS}.}\hskip 0.5em\ignorespaces}
   {\hfill $\Box$\par\medskip}

\begin{document}
\title[A refinement of a result of Andrews and Newman]{A refinement of a result of Andrews and Newman on the sum of minimal excludants} 

\author{Nayandeep Deka Baruah}
\address{Nayandeep Deka Baruah, Department of Mathematical Sciences, Tezpur University, Napaam, Assam-784028, India.}
\email{nayan@tezu.ernet.in, nayandeeptezu@gmail.com}

\author{Subhash Chand Bhoria}
\address{Subhash Chand Bhoria, Department of Mathematics, Pt. Chiranji Lal Sharma Government College, Sector-14 Karnal,  Haryana-132001, India.}
\email{scbhoria89@gmail.com}

\author{Pramod Eyyunni}
\address{Pramod Eyyunni, Department of Mathematics, Indian Institute of Technology Indore, Simrol, Indore,  Madhya Pradesh-453552,  India.} 
\email{pramodeyy@iiti.ac.in, pramodeyy@gmail.com}

\author{Bibekananda Maji}
\address{Bibekananda Maji, Department of Mathematics, Indian Institute of Technology Indore, Simrol, Indore, Madhya Pradesh-453552,  India.}
\email{bibekanandamaji@iiti.ac.in,  bibek10iitb@gmail.com}


\maketitle

\begin{abstract}
In this article,  we refine a result of Andrews and Newman, that is, the sum of minimal excludants over all the partitions of a number $n$ equals the number of partitions of $n$ into distinct parts with two colors. As a consequence, we find congruences modulo 4 and 8 for the functions appearing in this refinement. We also conjecture three further congruences for these functions. In addition, we also initiate the study of $k^{th}$ moments of minimal excludants. At the end, we also provide an alternate proof of a beautiful identity due to Hopkins, Sellers and Stanton.  
\end{abstract}

\section{Introduction}\label{intro}

The study of the minimal excludant of a partition has been an active area of research after being revived by the works of Andrews and Newman  \cite{geaMinexcl2019,  geaJISpaper2020}.  They defined
the minimal excludant of an integer partition $\pi$ as the least positive integer missing from the partition, denoted by $\text{mex}(\pi)$. For example, with $n=5$, the values of minimal excludants for the partitions of $n$ are: $\text{mex}(5)=1,\text{mex}(4+1)=2, \text{mex}(3+2)=1,\text{mex}(3+1+1)=2,\text{mex}(2+2+1)=3, \text{mex}(2+1+1+1)=3,\text{mex}(1+1+1+1+1)=2$. 
They discovered an elegant identity \cite [p.~250, Theorem 1.1]{geaMinexcl2019} involving the quantity $\sigma\text{mex}(n)$,  which denotes the sum of minimal excludants over all the partitions of $n$.  The identity is as follows:
\begin{align} 
\sum_{n=0}^{\infty}\sigma \text{mex}(n)q^n=(-q;q)^2_{\infty}=\sum_{n=0}^{\infty}D_2(n)q^n,  \label{AndrewsNewman Sigma Identity}
\end{align} 
where $D_2(n)$ represents the number of partitions of $n$ into distinct parts with two colors.
The same identity,  under a different terminology, was obtained earlier by Grabner and Knopfmacher \cite[p.~445, Equation (4.2)]{GrabnerKnopfmacher}. They studied the concept of `\emph{smallest gap}' in a partition that has exactly the same meaning as that of `minimal excludant'. 
The identity \eqref{AndrewsNewman Sigma Identity} has also been proved by Ballantine and Merca \cite{BallantineMerca}, employing purely combinatorial tools. For recent progress in this area, interested readers may refer to \cite{geaJISpaper2020,  barman-singh1,  barman-singh2,   Chakraborty-Ray,  sellers-dasilva, KBEM2021}. Here, and throughout the paper, $|q|<1$ and
\begin{align*}
(a;q)_0 :&=1, \qquad \\
(a;q)_n :&= (1-a)(1-aq)\cdots(1-aq^{n-1}),
\qquad n \geq 1, \\
(a;q)_{\i}:&= \lim_{n\to\i}(a;q)_n. 
\end{align*}
We also use the following notation:
\begin{itemize}

\item $\mathcal{P}(n):=$ the set of all partitions of $n$,

\item$p(n):=$ the number of partitions of $n$,

\item $\mathcal{D}_2^o(n):=$ the set of partitions of $n$ into distinct parts with two colors and an odd number of parts,

\item  $\mathcal{D}_2^e(n):=$ the set of partitions of $n$ into distinct parts with two colors and an even number of parts.

\end{itemize}

 Instead of summing the minimal excludants over all partitions of $n$, as done by Andrews and Newman in \cite{geaMinexcl2019}, we slightly modify the definition and consider two such sums parity-wise (over odd and even minimal excludants separately). 
\begin{definition} 
 For a positive integer $n$, we define the following functions:
\begin{align}
\sigma_o\textup{mex}(n):=\sum_{\substack{\pi\in \mathcal{P}(n)\\2\nmid\textup{mex}(\pi)}}\textup{mex}(\pi)\label{Sigma_odd_defn},
\\ \sigma_e\textup{mex}(n):=\sum_{\substack{\pi\in \mathcal{P}(n)\\2|\textup{mex}(\pi)}} \textup{mex}(\pi)\label{Sigma_even_defn}.
\end{align}
\end{definition}
For instance, with $n=5$,  one can observe that $\sigma_o\text{mex}(5)=1+1+3+3=8$ and $\sigma_e\text{mex}(5)=2+2+2=6$.

In this paper,  we refine the identity \eqref{AndrewsNewman Sigma Identity} by finding the generating functions for $ \sigma_o\textup{mex}(n)$ and $\sigma_e\textup{mex}(n)$.  

\section{Main Results}
In this section, we state our main results, the first of which connects the sum of minimal excludants considered parity wise with the number of partitions of $n$ into distinct parts with two colors,  according to the parity of their number of parts. We denote $D_2^e(n)$ (resp. $D_2^o(n)$) to be the number of partitions of $n$ into distinct parts with two colors and an even (resp. odd) number of parts. 
\begin{theorem}\label{main refinement} We have
\begin{align}
\sum_{n=0}^{\infty}\sigma_o\emph{mex}(n)q^n=\frac{(-q;q)^2_{\infty}+(q;q)^2_{\infty}}{2}=\sum_{n=0}^{\infty}D_2^e(n)q^n, \label{sigma_oGFN}
\end{align}
and
\begin{align}
\sum_{n=0}^{\infty}\sigma_e\emph{mex}(n)q^n=\frac{(-q;q)^2_{\infty}-(q;q)^2_{\infty}}{2}=\sum_{n=0}^{\infty}D_2^o(n)q^n.\label{sigma_eGFN}
\end{align}
\end{theorem}
\textbf{Example.}  In the tables below, we see that the identities \eqref{sigma_oGFN} and \eqref{sigma_eGFN} hold for $n=4$.
\begin{center}
\begin{tabular}{ cc }   
Table 1 & Table 2 \\  
\begin{tabular}{ |c|c|} 
\hline 
$\pi \in \mathcal{P}(4)$ & $\text{mex}(\pi)$\\
\hline
$4$ & $1$  \\

$3+1$ & $2$ \\

$2+2$ & $1$ \\

$2+1+1$ & $3$ \\

$1+1+1+1$ & $2$  \\
\hline
$\sigma_o\text{mex}(4)$ & $1+1+3=5$\\
\hline
$\sigma_e\text{mex}(4)$ & $2+2=4$ \\
\hline
\end{tabular} &  
\begin{tabular}{ |c|c| } 
\hline
$ \pi \in \mathcal{D}_2^o(4)$ & $\pi \in \mathcal{D}_2^e(4)$\\
\hline
$4_2$ & $3_2+1_2$  \\
$4_1$ & $3_2+1_1$\\
$2_2+1_2+1_1$ & $3_1+1_2 $\\
$2_1+1_2+1_1$ & $3_1+1_1$ \\
&  $2_2+2_1$\\
\hline
$D_2^o(4)=4$ & $D_2^e(4)=5$\\
\hline
\end{tabular} \\
\end{tabular}
\end{center}
As a consequence of Theorem \ref{main refinement}, we obtain the following congruences.
\begin{theorem}\label{congruence}
For any non-negative integer $n$,  
\begin{align}
\sigma_o \textup{mex}(2n+1) & \equiv 0 \pmod 4, \label{sigma o mod 4 cong} \\
\sigma_o \textup{mex}(4n+1) & \equiv 0 \pmod 8, \label{conj1} \\
 \sigma_e \textup{mex}(4n)  & \equiv 0 \pmod 4.  \label{econj1}
\end{align}
\end{theorem}
Based on computations evidence, we also propose a few more congruences for these functions. For more details see Conjecture \ref{main_conjecture} in Section \ref{concluding remarks}.

Now we state a proposition that we will use in several proofs.  As defined by Hopkins, Sellers, and Stanton \cite[Proposition 5]{HoSeStan2022},  let $x(m,n)$ be the number of partitions of $n$ whose minimal excludant is $m$. 
\begin{proposition}\label{master proposition}
Suppose that $n$ is a fixed non-negative integer.  Then for $z \in \mathbb{C}$, we have
\begin{align}
\sum_{m=1}^{\infty}x(m,n)z^m=p(n)+(z-1)\sum_{m=0}^{\infty}p\Big(n- \frac{m(m+1)}{2}\Big)z^m.\label{Master Identity}
\end{align}
\end{proposition}

\begin{remark}
We note that the sum on the right-hand side of the above result is in fact finite,  because by convention $p(n)=0$ for $n<0$.  Similarly,  the left-hand side is also a finite sum because for a fixed $n$ the minimal excludant of any partition of $n$ is trivially bounded by $n+1$.  
The comment also holds for all such `infinite-looking' partition sums in the sequel. 
\end{remark}

An immediate consequence of Proposition \ref{master proposition} arises by substituting $z=-1$  in \eqref{Master Identity}, which is a result due to Hopkins, Sellers and Stanton \cite [p.~5, Eq. (3)]{HoSeStan2022}, that is,
\begin{align*}
o(n)-e(n)=p(n)+2\sum_{m=1}^{\infty}(-1)^mp\left(n- \frac{m(m+1)}{2}\right),
\end{align*}
where $o(n)$ is the number of partitions of $n$ with an odd minimal excludant and $e(n)$ counts those with an even minimal excludant.  Andrews-Newmann \cite[Theorem 2]{geaJISpaper2020} and Hopkins-Sellers \cite[Theorem 1]{HopkinsSeller2020} independently showed that $o(n)$ equals the number of partitions of $n$ with non-negative crank.  This has been significantly generalized by Hopkins, Sellers and Stanton \cite [Theorem 2]{HoSeStan2022} recently.  

\textbf{The $k^{th}$ moments of minimal excludants:} We now study sums involving  $k^{th}$ powers of minimal excludants and generalize an identity of Andrews and Newman \cite[p. 251]{geaMinexcl2019}, which was crucial to their second proof of \eqref{AndrewsNewman Sigma Identity}.  
\begin{definition}
Let $k,  n \in \mathbb{N}$.  Then
\begin{align*}
\sigma^{(k)}\textup{mex}(n):&=\sum_{\pi\in \mathcal{P}(n)}\left(\textup{mex}(\pi)\right)^k,
\\ {\bar{\sigma}}^{(k)}\textup{mex}(n):&= \sum_{\pi\in \mathcal{P}(n)}(-1)^{\textup{mex}(\pi) -1}\left(\textup{mex}(\pi)\right)^k.
\end{align*}
\end{definition}
When $k=1$, we see that $\sigma^{(1)}\text{mex}(n) = \sigma \text{mex}(n)$ and we denote ${\bar{\sigma}}^{(1)}\text{mex}(n)$ by ${\bar{\sigma}}\text{mex}(n)$.  As a consequence of the next theorem we get exact formulae for $\sigma^{(k)}\text{mex}(n)$ and ${\bar{\sigma}}^{(k)}\text{mex}(n)$ in terms of partition functions.  

\begin{theorem}\label{kth moment of mex}
 For $k,n \in \mathbb{N}$ and $z \in \mathbb{C}$, the following identity holds true:
\begin{align}\label{Diferntial form of Master Identity}
\sum_{m=1}^{\infty} x(m,n)m^kz^m=\sum_{m=0}^{\infty}\left((m+1)^kz^{m+1}-m^kz^m\right)p\left(n- \frac{m(m+1)}{2}\right). 
\end{align}
\end{theorem}
In particular,  putting $z=1$ and $-1$, respectively, we have the following identities. 
\begin{corollary}
Given $k,  n \in \mathbb{N}$,  we have
\begin{equation}
\sigma^{(k)}{\rm mex}(n)=\sum_{m=0}^{\infty}\left((m+1)^k-m^k\right)p\left(n- \frac{m(m+1)}{2}\right), \label{moment_kth}
\end{equation}
\begin{equation}
{\bar{\sigma}}^{(k)}\emph{mex}(n)=\sum_{m=0}^{\infty}(-1)^m \left((m+1)^k+m^k\right)p\left(n- \frac{m(m+1)}{2}\right). \label{moment_bar}
\end{equation}
\end{corollary}

The corollary above also implies the next two  identities. The first one below arises in the second proof of the main result of Andrews and Newman on $\sigma\text{mex}(n)$ \cite[p. 251]{geaMinexcl2019} and the second is an analogous identity for $\bar{\sigma}\text{mex}(n)$.
\begin{corollary}\label{application_kth moment}
If $n$ is a positive integer, then 
\begin{align}
\sigma\emph{mex}(n)=\sum_{m=0}^{\infty}p\left(n- \frac{m(m+1)}{2}\right),\label{sigmamex_analogue}
\end {align}
\begin{align}
\bar{\sigma}\emph{mex}(n)=\sum_{m=0}^{\infty}(-1)^m(2m+1)p\left(n- \frac{m(m+1)}{2}\right).\label{sigmabarmex_analogue}
\end{align}
\end{corollary}
In the same spirit, we also define parity based sums involving $k^{th}$ powers of minimal excludants.
\begin{definition} For positive integers $k$ and $n$,  we define
\begin{align*}
\sigma_o^{(k)}\textup{mex}(n):=\sum_{\substack{\pi\in \mathcal{P}(n)\\2\nmid\textup{mex}(\pi)}}\left(\textup{mex}(\pi)\right)^k \text{and}
\quad \sigma_e^{(k)}\textup{mex}(n):=\sum_{\substack{\pi\in \mathcal{P}(n)\\2|\textup{mex}(\pi)}} \left(\textup{mex}(\pi)\right)^k. 
\end{align*}
\end{definition}
We have the following formulae. 
\begin{theorem}\label{kth moment_sigma_odd_even}
For positive integers $k$ and $n$, we have
\begin{align}
\sigma_o^{(k)}\emph{mex}(n)=\sum_{m=0}^{\infty}\left\{\delta_m(m+1)^k-(1-\delta_m)m^k\right\}p\left(n- \frac{m(m+1)}{2}\right) \label{General sigma odd formula}
\end{align}
and
\begin{align}
\sigma_e^{(k)}\emph{mex}(n)=\sum_{m=0}^{\infty}\left\{(1-\delta_m)(m+1)^k-\delta_mm^k\right\}p\left(n- \frac{m(m+1)}{2}\right),\label{General sigma even formula}
\end{align}
where the function $\delta_m$ is defined as 
\begin{align*}
\delta_m:=\frac{1+(-1)^m}{2}=\begin{cases}
1, & \text{if m is even},
\\0, & \text{otherwise}.
\end{cases}
\end{align*}
\end{theorem}
We mention an immediate corollary of the above theorem that gives explicit formulae for the functions $\sigma_o^{(1)}\text{mex}(n)$ and $\sigma_e^{(1)} \text{mex}(n)$, which are nothing but $\sigma_o{\rm mex}(n)$ and $\sigma_e{\rm mex}(n)$, respectively,  as defined in \eqref{Sigma_odd_defn} and \eqref{Sigma_even_defn}. 

\begin{corollary}\label{sigma_odd_even_mex} 
Let $t_m=\frac{m(m+1)}{2}$ be the $m^{th}$ triangular number.  
For $n \in \mathbb{N}$, one has
\begin{align}
\sigma_o\emph{mex}(n)&=\sum_{m=0}^{\infty}(2m+1)\left\{p\left(n- t_{2m}\right)-p\left(n- t_{2m+1}\right)\right\}\label{sigma odd formula}
\\ &=\left\{p(n)-p(n-1)\right\}+3\left\{p(n-3)-p(n-6)\right\}+\cdots \nonumber
\end{align}
and
\begin{align}
\sigma_e\emph{mex}(n)&=\sum_{m=0}^{\infty}(2m+2)\left\{p\left(n- t_{2m+1}\right)-p\left(n- t_{2m}\right)\right\}\label{sigma even formula}
\\&=2\left\{p(n-1)-p(n-3)\right\}+4\left\{p(n-6)-p(n-10)\right\}+\cdots.  \nonumber
\end{align}
\end{corollary} 

Recall that $o(n)$ is the number of partitions of $n$ with an odd minimal excludant. Denote by $o_1(n)$ (resp. $o_3(n)$) the number of partitions of $n$ with minimal excludant congruent to $1$ modulo $4$ (resp. $3$ modulo 4). So, $o(n) =o_1(n) +o_3(n)$. Also, let $q(n)$ represent the number of partitions of $n$ into distinct parts.  Our results also lead to an alternative proof for the following result of Hopkins, Sellers and Stanton \cite[Proposition 8]{HoSeStan2022}.
\begin{proposition}\label{HSS}
For $n\geq 1$, 
\begin{align*}
o_1(n)=\begin{cases}
o_3(n), & \text{if n is odd},
\\o_3(n)+q(n/2), & \text{otherwise}.
\end{cases}
\end{align*}
\end{proposition}
In \cite[Proposition 3.2]{HopSellYee2021}, Hopkins, Sellers and Yee proved the above result by analytic methods. 
 
We organize the rest of the paper in the following way. In the next section, we collect some useful results which will be essential for our proofs. Section \ref{proofs} is devoted to the proofs of the main results. In Section \ref{sec-five}, we present an alternative proof of Proposition \ref{HSS}. Finally, in the last section, we enunciate some open problems arising from this work. 

\section{Preliminaries}

We first recall some standard results on Ramanujan's theta function $f(a,b)$, which is defined, for $|ab|<1$, by
\begin{equation*}
f(a, b):= \sum_{n=-\infty}^{\infty} a^{\frac{n(n+1)}{2}} b^{\frac{n(n-1)}{2}}. 
\end{equation*} 
The famous triple product identity due to Jacobi is given by \cite[p. 35, Entry 19]{RNBIII}
\begin{equation}\label{JTPI}
f(a, b) = (-a; ab)_{\infty} (-b; ab)_{\infty} (ab; ab)_{\infty}.
\end{equation}
In the sequel,  we use the notation $f_k:=(q^k;q^k)_\infty$, $k\ge1$.
As special cases of \eqref{JTPI},  we have
\begin{align} \label{JTP for phi(q)}
\varphi(q):&= f(q,q)= \sum_{n=-\infty}^\infty q^{n^2}= (-q;q^2)_\infty^2 (q^2; q^2)_\infty=\dfrac{f_2^5}{f_1^2f_4^2}\\\intertext{and}
\psi(q):&= f(q,q^3)=\sum_{m=0}^\infty q^{m+1\choose2}=(-q;q^4)_\infty(-q^3;q^4)_\infty(q^4;q^4)_\infty=\dfrac{f_2^2}{f_1}. \label{Ramatheta}  
\end{align}
Replacing $q$ by $-q$ in \eqref{JTP for phi(q)}, we also have
\begin{align} \label{JTP for phi(-q)}
\varphi(-q)&= (q;q^2)_\infty^2 (q^2; q^2)_\infty=\dfrac{f_1^2}{f_2}. 
\end{align}
Now, from \cite[p. 40, Entry 25]{RNBIII}, we note that
\begin{align}\label{phiplus}
\varphi(q)&+\varphi(-q)=2\varphi(q^4)\\\intertext{and}
\label{phiminus}\varphi(q)&-\varphi(-q)=4q\psi(q^8).
\end{align}
Therefore, 
\begin{align*}
\varphi(q)=\varphi(q^4)+2q\psi(q^8), \quad \text{and}\quad
\varphi(-q)=\varphi(q^4)-2q\psi(q^8).
\end{align*} 
Using the above two identities and invoking \eqref{JTP for phi(q)}, \eqref{Ramatheta} and \eqref{JTP for phi(-q)},  we readily get the following  2-dissection formulas for $1/f_1^2$ and $f_1^{2}$: 
\begin{align}\label{reci}
\frac{1}{f_1^2}& = \frac{f_8^5}{f_2^5 f_{16}^2} +2q \frac{f_4^2 f_{16}^2}{f_2^5 f_8},\\
\label{f_1^2}f_1^2 & = \frac{f_2 f_8^5}{f_4^2 f_{16}^2} - 2q \frac{f_2 f_{16}^2}{f_8}. 
\end{align}
We also recall the following 2-dissection of $\psi(q)$ from \cite[p. 49, Corollary (ii)]{RNBIII}:
\begin{equation}\label{psi 2diss}
\psi(q) = f(q^6, q^{10}) + q f(q^2, q^{14}).
\end{equation}
We end this section with another important $q$-series identity namely, Jacobi's identity \cite[p. 39, Entry 24(ii)]{RNBIII}:
\begin{align}
\sum_{m=0}^{\infty}(2m+1)(-1)^mq^{m+1\choose2}&=(q;q)_{\infty}^3. \label{Jacobi}
\end{align}

The next section is devoted to proving our main results. 
\section{Proofs of main results}\label{proofs}
\begin{proof}[Theorem \rm{\ref{main refinement}}][]
 Let $x(2m+1, n)$ denote the number of partitions of $n$ whose minimal excludant is $2m+1$. Thus, we have
\begin{equation}
\sum_{n=0}^{\infty} x(2m+1,n)q^n=\frac{\displaystyle q^{1+2+\cdots+2m}}{\displaystyle \prod_{ {\substack{r=1 \\ r \neq {2m+1}}}}^{\infty}(1-q^r)}=\frac{q^{{2m+1}\choose2}(1-q^{2m+1})}{(q;q)_{\infty}}.\label{ODDminexcl}
\end{equation}
Now we introduce a two-parameter function $L(z,q)$ with complex variables $z$ and $q$, where the exponent of $z$ keeps track of the odd minimal excludant $2m+1$. Thus, 
\begin{align}
L(z,q):= \sum_{n=0}^{\infty}\sum_{m=0}^{\infty}  x(2m+1,n)z^{2m+1}q^n  & = \sum_{m=0}^{\infty}z^{2m+1}\sum_{n=0}^{\infty} x(2m+1,n)q^n\nonumber\\ &=\frac{1}{(q;q)_{\infty}}\sum_{m=0}^{\infty}z^{2m+1}q^{{2m+1}\choose2}(1-q^{2m+1}), \label{L(z,q)}
\end{align}
where we used \eqref{ODDminexcl} in the final step.  Upon differentiating $L(z,q)$ with respect to $z$ and putting $z=1$, we get
\begin{align}
\frac{\partial}{\partial z}L(z,q)\Bigg|_{z=1}
&=\sum_{n=0}^{\infty}\Bigg(\sum_{m=0}^{\infty}(2m+1) x(2m+1,n)\Bigg)q^n.\label{ODDpartial derivative term}
\end{align}
Observe that the inner sum of the right hand side above is
\begin{align}
\sum_{m=0}^{\infty}(2m+1) x(2m+1,n)=\sum_{\substack{\pi\in \mathcal{P}(n)\\2\nmid\text{mex}(\pi)}}\text{mex}(\pi)=\sigma_o\text{\text{mex}}(n).\label{sigmaODD expression}
\end{align}
After making use of \eqref{sigmaODD expression}, \eqref{L(z,q)} in \eqref{ODDpartial derivative term}, we obtain
\begin{align}
\sum_{n=0}^{\infty}\sigma_o\text{mex}(n)q^n&= \Bigg[ \frac{\partial}{\partial z} \left( \frac{1}{(q;q)_{\infty}}\sum_{m=0}^{\infty}z^{2m+1}q^{{2m+1}\choose2} \left(1-q^{2m+1} \right) \right)\Bigg]_{z=1}\nonumber
\\&=\frac{1}{(q;q)_{\infty}}\sum_{m=0}^{\infty}(2m+1)q^{{2m+1}\choose2} \left(1-q^{2m+1} \right)\nonumber
\\&=\frac{1}{(q;q)_{\infty}}\Bigg(\sum_{m=0}^{\infty}(2m+1)q^{2m+1\choose2}-\sum_{m=0}^{\infty}(2m+1)q^{2m+2\choose2}\Bigg)\nonumber
\\&=\frac{1}{(q;q)_{\infty}}\Bigg(\sum_{m=0}^{\infty}(2m+1)q^{2m+1\choose2}-\sum_{m=0}^{\infty}(2m+2)q^{2m+2\choose2}+\sum_{m=0}^{\infty}q^{2m+2\choose2}\Bigg)\nonumber
\\&=\frac{1}{(q;q)_{\infty}}\Bigg(\sum_{m=0}^{\infty}(m+1)(-1)^mq^{m+1\choose2}+\sum_{m=0}^{\infty}q^{2m+2\choose2}\Bigg).\label{Step before dividing 2}
\end{align}
Now multiplying and dividing by $2$ in the right hand side of the last equality \eqref{Step before dividing 2}, and then simplifying yields,
{\allowdisplaybreaks \begin{align}
&\frac{1}{2(q;q)_{\infty}}\Bigg(\sum_{m=0}^{\infty}(2m+2)(-1)^mq^{m+1\choose2}+\sum_{m=0}^{\infty}2q^{2m+2\choose2}\Bigg)\nonumber
\\&=\frac{1}{2(q;q)_{\infty}}\Bigg(\sum_{m=0}^{\infty}(2m+1)(-1)^mq^{m+1\choose2}+\sum_{m=0}^{\infty}(-1)^mq^{m+1\choose2}+\sum_{m=0}^{\infty}2q^{2m+2\choose2}\Bigg)\nonumber
\\&=\frac{1}{2(q;q)_{\infty}}\Bigg(\sum_{m=0}^{\infty}(2m+1)(-1)^mq^{m+1\choose2}+\sum_{m=0}^{\infty}q^{2m+1\choose2}-\sum_{m=0}^{\infty}q^{2m+2\choose2}+\sum_{m=0}^{\infty}2q^{2m+2\choose2}\Bigg)\nonumber
\\&=\frac{1}{2(q;q)_{\infty}}\Bigg(\sum_{m=0}^{\infty}(2m+1)(-1)^mq^{m+1\choose2}+\sum_{m=0}^{\infty}q^{2m+1\choose2}+\sum_{m=0}^{\infty}q^{2m+2\choose2}\Bigg)\nonumber
\\&=\frac{1}{2(q;q)_{\infty}}\Bigg(\sum_{m=0}^{\infty}(2m+1)(-1)^mq^{m+1\choose2}+\sum_{m=0}^{\infty}q^{m+1\choose2}\Bigg).  \label{Dividingby2}
\end{align}}
Utilizing $\eqref{Dividingby2}$  in \eqref{Step before dividing 2}, we get
\begin{align}
\sum_{n=0}^{\infty}\sigma_o\text{mex}(n)q^n&=\frac{1}{2(q;q)_{\infty}}\Bigg(\sum_{m=0}^{\infty}(2m+1)(-1)^mq^{m+1\choose2}+\sum_{m=0}^{\infty}q^{m+1\choose2}\Bigg). \label{Step after dividing 2}
\end{align}
We now invoke \eqref{Jacobi} and \eqref{Ramatheta} and substitute the product side of these identities for the two series in the right hand side of \eqref{Step after dividing 2} to obtain
\begin{align}
\sum_{n=0}^{\infty}\sigma_o\text{mex}(n)q^n =\frac{1}{2(q;q)_{\infty}}\Bigg((q;q)_{\infty}^3+\frac{(q^2;q^2)_{\infty}}{(q;q^2)_{\infty}}\Bigg) &=\frac{1}{2}\Bigg(\frac{(q;q)_{\infty}^3}{(q;q)_{\infty}}+\frac{(q^2;q^2)_{\infty}}{(q;q)_{\infty}(q;q^2)_{\infty}}\Bigg)\nonumber
\\&=\frac{1}{2}\Big((q;q)_{\infty}^2+(-q;q)_{\infty}^2\Big), \nonumber
\end{align}
where the last equality follows from Euler's identity, i.e.,  $(-q;q)_\infty = 1/(q;q^2)_\infty$.  
Observe that the last expression above can be written as
\begin{align}\label{even_distinct_twocolor}
(-q;q)^2_{\infty}+(q;q)^2_{\infty} & =\sum_{n=0}^{\infty}\left(D_2^e(n)+D_2^o(n)\right)q^n+\sum_{n=0}^{\infty}\left(D_2^e(n)-D_2^o(n)\right)q^n\nonumber  \\
 & =2\sum_{n=0}^{\infty}D_2^e(n)q^n.
\end{align}
So we conclude that \eqref{sigma_oGFN} holds true. Finally, subtracting \eqref{sigma_oGFN} from \eqref{AndrewsNewman Sigma Identity}, we obtain \eqref{sigma_eGFN}.
\end{proof}

\noindent \textbf{An alternative proof of Theorem  \ref{main refinement}:}
First, we state and prove a lemma. 
\begin {lemma}\label{sigmabar gfn}
We have 
\begin{equation}\label{sigmasigmabar}
\sum_{n=0}^{\infty}\bar{\sigma}\emph{mex}(n)q^n=(q; q)^2_{\infty}.
\end{equation}
\end{lemma}

\begin{proof} Let $x(m, n)$ be the number of partitions of $n$ whose minimal excludant is $m$.   Proceeding analogously as in the previous proof from \eqref{ODDminexcl} to \eqref{L(z,q)},  we can show that 
\begin{align}
M(z,q):= \sum_{n=0}^{\infty}\sum_{m=1}^{\infty}  x(m,n)z^m q^n 
=\frac{1}{(q;q)_{\infty}}\sum_{m=1}^{\infty}z^mq^{m\choose2}(1-q^m).  \label{M(z,q)}
\end{align}
 Upon differentiating $M(z, q)$ with respect to $z$ and substituting $z=-1$,  we arrive at
\begin{equation}
\frac{\partial}{\partial z}M(z,q)\Bigg|_{z=-1}
=\sum_{n=0}^{\infty}\Bigg(\sum_{m=1}^{\infty}m(-1)^{m-1}  x(m,n)\Bigg)q^n.\label{partial derivative term}
\end{equation}
The inner sum in the right hand side above is actually
\begin{align*}
\sum_{m=1}^{\infty}m(-1)^{m-1}  x(m,n)= \sum_{\pi\in \mathcal{P}(n)}(-1)^{\text{mex}(\pi) -1 }\text{mex}(\pi)=\bar{\sigma}\text{\text{mex}}(n). 
\end{align*}
Putting this back in \eqref{partial derivative term} gives us
\begin{equation}\label{auxiliaryid}
\sum_{n=0}^{\infty}\bar{\sigma}\text{mex}(n)q^n = \frac{\partial}{\partial z}M(z,q)\Bigg|_{z=-1}.
\end{equation}
We now compute the right-hand side of the above by invoking the expression for $M(z, q)$ in \eqref{M(z,q)},  that is,  
\begin{align*}
& \Bigg[\frac{\partial}{\partial z}\left(  \frac{1}{(q;q)_{\infty}}\sum_{m=1}^{\infty}z^mq^{m\choose2}(1-q^m) \right) \Bigg]_{z=-1} \\
 &=\frac{1}{(q;q)_{\infty}}\sum_{m=1}^{\infty}m(-1)^{m-1}q^{m\choose2}(1-q^m) \\
&=\frac{1}{(q;q)_{\infty}}\Bigg(\sum_{m=1}^{\infty}m(-1)^{m-1}q^{m\choose2}-\sum_{m=1}^{\infty}m(-1)^{m-1}q^{{m\choose2}+m}\Bigg) \\
&=\frac{1}{(q;q)_{\infty}}\Bigg(\sum_{m=0}^{\infty}(m+1)(-1)^mq^{m+1\choose2}-\sum_{m=0}^{\infty}m(-1)^{m-1}q^{{m+1\choose2}}\Bigg) \\
&=\frac{1}{(q;q)_{\infty}}\sum_{m=0}^{\infty}(2m+1)(-1)^mq^{m+1\choose2}\\
&=(q;q)_\infty^2,
\end{align*}
where we used Jacobi's identity \eqref{Jacobi} in the last equality. Putting this in \eqref{auxiliaryid} gives us the desired identity \eqref{sigmasigmabar}. 
\end{proof}
We are now in a position to present an alternative proof of Theorem \ref{main refinement}. 
\begin{proof}[Theorem \rm{\ref{main refinement}}][]
To prove identity \eqref{sigma_oGFN} in Theorem \ref{main refinement}, we make use of Lemma \ref{sigmabar gfn} and \eqref{AndrewsNewman Sigma Identity}.  Adding  \eqref{AndrewsNewman Sigma Identity} and \eqref{sigmasigmabar} together gives us
\begin{align}
\sum_{n=0}^{\infty}\left(\sigma \text{mex}(n)+\bar{\sigma}\text{mex}(n)\right)q^n=(-q;q)^2_{\infty}+(q;q)^2_{\infty}\label{sigma+sigmabar}.
\end{align}
Here, note that the summand on the left hand side can be written as 
\begin{align*}
\sigma \text{mex}(n)+\bar{\sigma}\text{mex}(n) &  =\sum_{\pi \in \mathcal{P}( n)}\text{mex}(\pi)-\sum_{\pi \in \mathcal{P}( n)}(-1)^{\text{mex}(\pi)}\text{mex}(\pi) \\
 &=2\sum_{\substack{\pi\in \mathcal{P} (n)\\2\nmid\text{mex}(\pi)}}\text{mex}(\pi)=2\sigma_o\text{mex}(n).
\end{align*}
Substituting this value into \eqref{sigma+sigmabar}, we find that
\begin{align*}
\sum_{n=0}^{\infty}\sigma_o\text{mex}(n)q^n=\frac{(-q;q)^2_{\infty}+(q;q)^2_{\infty}}{2}. 
\end{align*}
We have already interpreted the sum on the right-hand side of the above equation in \eqref{even_distinct_twocolor}.  Thus, we conclude that \eqref{sigma_oGFN} holds true and the proof of \eqref{sigma_eGFN} follows from \eqref{sigma_oGFN} and \eqref{AndrewsNewman Sigma Identity}. 
\end{proof}

\begin{proof}[Theorem {\rm \ref{congruence}}][] 
From Theorem \ref{main refinement}, we have
\begin{equation}\label{gfn sigma o}
\sum_{n=0}^{\infty} \sigma_o \textup{mex} (n) q^n = \frac{(-q; q)^2_{\infty} + (q; q)^2_{\infty}}{2},
\end{equation}
which  can be recast as
\begin{equation} \label{pre 2-diss}
\sum_{n=0}^{\infty} \sigma_o \textup{mex} (n) q^n = \frac{1}{2} \left( \frac{f_2^2}{f_1^2} + f_1^2\right).
\end{equation}
Substituting  the $2$-dissections of $f_1^{2}$ and $1/f_1^{2}$ from  \eqref{f_1^2} and \eqref{reci} into \eqref{pre 2-diss}, we have
\begin{equation} \label{pre-extract}
\sum_{n=0}^{\infty} \sigma_o \textup{mex} (n) q^n = \frac{1}{2} \left( \frac{f_8^5}{f_2^3 f_{16}^2} + 2q \frac{f_4^2 f_{16}^2}{f_2^3 f_8} + \frac{f_2 f_8^5}{f_4^2 f_{16}^2} -2q \frac{f_2 f_{16}^2}{f_8} \right).
\end{equation}
Extracting the odd powers on both sides of \eqref{pre-extract}, that is, equating the sub-series consisting of odd powers of $q$ on both sides of \eqref{pre-extract}, we obtain
\begin{equation*}
\sum_{n=0}^{\infty} \sigma_o \textup{mex} (2n+1) q^{2n+1} = q \left( \frac{f_4^2 f_{16}^2}{f_2^3 f_8} - \frac{f_2 f_{16}^2}{f_8} \right).
\end{equation*}
Dividing both sides of the above by $q$ and then replacing $q^2$ by $q$, we arrive at 
\begin{equation} \label{gfn sigma o odd}
\sum_{n=0}^{\infty} \sigma_o \textup{mex} (2n+1) q^{n} = \frac{f_2^2 f_{8}^2}{f_1^3 f_4} - \frac{f_1 f_{8}^2}{f_4}.
\end{equation}

Now we prove that $f_1^4 \equiv f_2^2 \pmod{4}$. Note that
\begin{align*}
f_1^4 = \prod_{r=1}^{\infty} (1 - q^r)^4\quad\textup{and}\quad f_2^2 = \prod_{r=1}^{\infty} (1 - q^{2r})^2.
\end{align*}
But,
\begin{equation*}
(1-q^r)^4 = 1-4q^r + 6q^{2r}-4q^{3r}+q^{4r} \equiv 1-2q^{2r} + q^{4r} = (1-q^{2r})^2 \pmod{4}.
\end{equation*}
Therefore,
\begin{align*}
f_1^4 \equiv f_2^2 \pmod{4}.
\end{align*}
Employing the above congruence in \eqref{gfn sigma o odd}, we find that
\begin{align*} 
&\sum_{n=0}^{\infty} \sigma_o \textup{mex} (2n+1) q^{n} = \frac{f_1 f_{8}^2}{f_4} \left( \frac{f_2^2}{f_1^4} - 1 \right)\equiv 0 \pmod{4}.
\end{align*}
Thus,
\begin{align*} 
\sigma_o \textup{mex} (2n+1) \equiv 0 \pmod{4},
\end{align*}
which completes the proof of \eqref{sigma o mod 4 cong}. 

Now we move to the proof of the second congruence of the theorem,  namely,  \eqref{conj1}.
To this end,  we need to find a $2$-dissection of the generating function \eqref{gfn sigma o odd} of $\sigma_o \textup{mex} (2n+1)$. We can rewrite \eqref{gfn sigma o odd} as
\begin{align*}
\sum_{n=0}^{\infty} \sigma_o \textup{mex} (2n+1) q^{n} &= \frac{f_8^2}{f_4}  \left( \frac{1}{f_1^2} - \frac{f_1^2}{f_2^2} \right) \psi(q).
\end{align*}
Employing \eqref{f_1^2}, \eqref{reci} and \eqref{psi 2diss} in the above, we have
\begin{align*}
&\sum_{n=0}^{\infty} \sigma_o \textup{mex} (2n+1) q^{n} \notag\\
&= \frac{f_8^2}{f_4} \left( \frac{f_8^5}{f_2^5 f_{16}^2} + 2q \frac{f_4^2 f_{16}^2}{f_2^5 f_8} - \frac{f_8^5}{f_2 f_4^2 f_{16}^2} +2q \frac{f_{16}^2}{f_2 f_8} \right) \left(f(q^6, q^{10}) + q f(q^2, q^{14})\right), 
\end{align*}
which is a 2-dissection of the generating function  of $\sigma_o \textup{mex} (2n+1)$. Extracting the even terms from both sides of the above and then replacing $q^2$ by $q$, we find that 
\begin{align}
\sum_{n=0}^{\infty} \sigma_o \textup{mex} (4n+1) q^{n} &= \frac{f_4^2}{f_2} \left(   f(q^3, q^{5}) \left( \frac{f_4^5}{f_1^5 f_{8}^2} - \frac{f_4^5}{f_1 f_2^2 f_{8}^2} \right) + 2q f(q, q^{7}) \left( \frac{f_2^2 f_{8}^2}{f_1^5 f_4}  +  \frac{f_{8}^2}{f_1 f_4}\right)\right) \nonumber \\
&= f(q^3, q^{5}) \frac{f_4^7}{f_1 f_2^3 f_8^2} \left( \frac{f_2^2}{f_1^4} - 1 \right) + 2q f(q, q^7) \frac{f_4 f_8^2}{f_1 f_2} \left( \frac{f_2^2}{f_1^4} + 1 \right). \label{N.1}
\end{align}
To simplify the above, we transform the terms involving $f(q, q^7)$ into terms involving $f(q^3, q^5)$. To that end, first we recall from \cite[p. 51, Example (iv)]{RNBIII} that
\begin{align*}
2f^2(q^3, q^5) &= \psi(q) \varphi(q^2) + \psi(-q) \varphi(-q^2),  \\
2q f^2(q, q^7) &= \psi(q) \varphi(q^2) - \psi(-q) \varphi(-q^2). 
\end{align*}
The above two imply that
\begin{align}
f^2(q^3, q^5) + q f^2(q, q^7) &= \psi(q) \varphi(q^2), \label{N.4} \\
f^2(q^3, q^5) - q f^2(q, q^7) &= \psi(-q) \varphi(-q^2). \label{N.5} 
\end{align}
Now, with the aid of Jacobi's triple product identity \eqref{JTPI}, we have
\begin{align}
f(q^3, q^{5}) f(q, q^{7}) &= (-q, -q^3, -q^5, -q^7; q^8)_{\infty} f_8^2 \nonumber \\
&= (-q; q^2)_{\infty} f_8^2. \label{N.6}
\end{align}
Dividing both sides of \eqref{N.4} by $f(q^3, q^{5}) f(q, q^{7})$ and then using \eqref{N.6}, we obtain
\begin{equation} \label{N.7}
\frac{f(q^3, q^5)}{f(q, q^7)} + q \frac{f(q, q^7)}{f(q^3, q^5)} = \frac{\psi(q) \varphi(q^2)}{(-q; q^2)_{\infty} f_8^2}=\frac{\varphi(q^2)}{\psi(q^4)}.
\end{equation}
Again, dividing both sides of \eqref{N.5} by $f(q^3, q^{5}) f(q, q^{7})$, followed by a second appeal to \eqref{N.6} and subsequent simplifications give us
\begin{equation}\label{N.8}
\frac{f(q^3, q^5)}{f(q, q^7)} - q \frac{f(q, q^7)}{f(q^3, q^5)} = \frac{\psi(-q) \varphi(-q^2)}{(-q; q^2)_{\infty} f_8^2} = \frac{\varphi(-q)}{\psi(q^4)}. 
\end{equation}
Subtracting \eqref{N.8} from \eqref{N.7}, we have
\begin{align} \label{N.9}
2q \frac{f(q, q^7)}{f(q^3, q^5)} = \frac{\varphi(q^2) - \varphi(-q)}{\psi(q^4)},\notag\\\intertext{that is,}
2q f(q, q^7) = f(q^3, q^5) \frac{\varphi(q^2) - \varphi(-q)}{\psi(q^4)}.
\end{align}
Employing \eqref{N.9} in \eqref{N.1}, and then using the $q$-product representations of $\varphi(q^2)$, $\varphi(-q)$ and $\psi(q^4)$ from \eqref{JTP for phi(q)} \eqref{JTP for phi(-q)} and \eqref{Ramatheta}, we obtain
\begin{align}\label{gen-4n1}
&\sum_{n=0}^{\infty} \sigma_o \textup{mex} (4n+1) q^n \notag\\
&= f(q^3, q^5) \frac{f_4^7}{f_1 f_2^3 f_8^2} \left( \frac{f_2^2}{f_1^4} - 1 \right) + f(q^3, q^5) \frac{f_4 f_8^2}{f_1 f_2} \left( \frac{f_2^2}{f_1^4} + 1 \right) \left( \frac{\varphi(q^2) - \varphi(-q)}{\psi(q^4)}\right) \notag \\
&= f(q^3, q^5) \frac{f_4^2}{f_1^3} \left( 2 \frac{f_4^5}{f_1^2 f_2 f_8^2} - \frac{f_1^4}{f_2^2} - 1 \right). 
\end{align}
Noting that $2 f_1^4 \equiv 2 f_2^2 \pmod{8}$, we infer from the above that
\begin{align}
\sum_{n=0}^{\infty} \sigma_o \textup{mex} (4n+1) q^n &\equiv f(q^3, q^5) \frac{f_4^2}{f_1^3} \left( 2 \frac{f_1^2 f_4}{f_2^3} - \frac{f_1^4}{f_2^2} - 1 \right) \notag\\
&\equiv f(q^3, q^5) \frac{f_4^2}{f_1^3} \left( 2 \frac{f_1^2 f_2}{f_4} - \frac{f_1^4}{f_2^2} - 1 \right)\notag\\
&\equiv f(q^3, q^5) \frac{f_4^2}{f_1^3} \left(2 \varphi(-q) \varphi(-q^2) - \varphi^2 (-q) - 1\right) \pmod{8}.\label{N.10}
\end{align}
Now, from \eqref{phiminus}, we have 
\begin{align}\label{phimp}
\varphi(-q) \equiv \varphi(q) \pmod{4}.
\end{align}
Multiplying \eqref{phiplus} and \eqref{phiminus}, we also have
 \begin{align*}
\varphi^2(q) -\varphi^2(-q) =8q\varphi(q^4) \psi(q^8), 
\end{align*}
which implies that
\begin{align}\label{phisq-mp}
\varphi^2(-q) \equiv \varphi^2(q) \pmod{8}.
\end{align}
With the aid of \eqref{phimp} and \eqref{phisq-mp}, we can recast \eqref{N.10} as 
\begin{align}
\sum_{n=0}^{\infty} \sigma_o \textup{mex} (4n+1) q^n 
&\equiv f(q^3, q^5) \frac{f_4^2}{f_1^3} \left(2 \varphi(q) \varphi(q^2) - \varphi^2 (q) - 1\right)\pmod{8}.\label{N.10A}
\end{align}
Hence from the above identity, to prove \eqref{conj1}, it is enough to show that
\begin{equation} \label{N.12}
2 \varphi(q) \varphi(q^2) - \varphi^2 (q) - 1 \equiv 0 \pmod{8}.
\end{equation}
To prove \eqref{N.12}, we recall the Lambert series representations of $\varphi^2(q)$ and $\varphi(q) \varphi(q^2)$ from \cite[ p. 58, 
Eq. (3.2.9); p. 73, Eq. (3.7.3)]{bcbSpirit}, namely,
\begin{align}
\varphi^2(q) &= 1 + 4 \sum_{n=1}^{\infty} \frac{q^n}{1+q^{2n}}, \label {lambert1} \\
\varphi(q) \varphi(q^2) &= 1 + 2 \sum_{n=1}^{\infty} \frac{q^n + q^{3n}}{1 + q^{4n}}. \label{lambert2}
\end{align}
From \eqref{lambert1} and \eqref{lambert2}, we see that
\begin{align*}
2 \varphi(q) \varphi(q^2) - \varphi^2 (q) - 1 &= 4 \left( \sum_{n=1}^{\infty} \frac{q^n + q^{3n}}{1 + q^{4n}} - \sum_{n=1}^{\infty} \frac{q^n}{1+q^{2n}} \right)  \\
&\equiv 4 \left( \sum_{n=1}^{\infty} \frac{q^n (1 + q^{2n})}{1- q^{4n}} - \sum_{n=1}^{\infty} \frac{q^n}{1 + q^{2n}} \right)  \\
&\equiv 4 \left( \sum_{n=1}^{\infty} \frac{q^n}{1 - q^{2n}} - \sum_{n=1}^{\infty} \frac{q^n}{1 - q^{2n}} \right) \equiv 0 \pmod{8}.
\end{align*}
Employing the above in \eqref{N.10A}, we readily arrive at \eqref{conj1}. 

Next,  we give the proof of the last congruence,  namely,  \eqref{econj1}.
Proceeding as in  \eqref{gfn sigma o}--\eqref{gfn sigma o odd}, but with 
\begin{equation*}
\sum_{n=0}^{\infty} \sigma_e \textup{mex}(n) q^n = \frac{(-q; q)_{\infty}^2 - (q; q)_{\infty}^2}{2}
\end{equation*}
instead of 
\begin{equation*}
\sum_{n=0}^{\infty} \sigma_o \textup{mex}(n) q^n = \frac{(-q; q)_{\infty}^2 + (q; q)_{\infty}^2}{2},
\end{equation*}
we find that
\begin{equation}\label{N.13}
\sum_{n=0}^{\infty} \sigma_e \textup{mex}(2n) q^n = \frac{1}{2} \left( \frac{f_4^5}{f_1^3 f_8^2} - \frac{f_1 f_4^5}{f_2^2 f_8^2} \right) = \frac{1}{2} \frac{f_4^6}{f_2^2 f_8^4} \left( \frac{f_2^2 f_8^2}{f_1^3 f_4} - \frac{f_1 f_8^2}{f_4} \right). 
\end{equation}
From \eqref{N.13} and \eqref{gfn sigma o odd}, it follows that
\begin{equation*}
2 \sum_{n=0}^{\infty} \sigma_e \textup{mex} (2n) q^n = \frac{f_4^6}{f_2^2 f_8^4}\sum_{n=0}^{\infty} \sigma_o \textup{mex} (2n+1) q^n.
\end{equation*}
Extract the even powers of $q$ in the above equation and then replace $q^2$ by $q$ to see that 
\begin{equation}\label{N.14}
2 \sum_{n=0}^{\infty} \sigma_e \textup{mex} (4n) q^n = \frac{f_2^6}{f_1^2 f_4^4} \sum_{n=0}^{\infty} \sigma_o \textup{mex} (4n+1) q^n.
\end{equation}
Since by \eqref{conj1}, we have already shown that $\sigma_o \textup{mex} (4n+1) \equiv 0 \pmod{8}$, it is immediate from \eqref{N.14} that $\sigma_e \textup{mex} (4n) \equiv 0 \pmod{4}$, which is \eqref{econj1}. 
\end{proof}

We now give an alternative proof of \eqref{sigma o mod 4 cong} using Proposition \ref{HSS}. 

\noindent\emph{Alternative proof of \eqref{sigma o mod 4 cong}.}
 For any positive integer $m$, we have
\begin{align*}
\sigma_o \textup{mex}(m) &= \sum_{\substack{\pi \in \mathcal{P}(m) \\ \textup{mex}(\pi) \equiv 1 \pmod{4}}} \textup{mex}(\pi) + \sum_{\substack{\pi \in \mathcal{P}(m) \\ \textup{mex}(\pi) \equiv 3 \pmod{4}}} \textup{mex}(\pi) \\
& \equiv \underbrace{1 + 1 + \cdots + 1}_{\text{$o_1(m)$ times}}
+ 
\underbrace{3 + 3 + \cdots + 3}_{\text{$o_3(m)$ times}}
\pmod{4} \\
& \equiv \underbrace{1 + 1 + \cdots + 1}_{\text{$o_1(m)$ times}}
+ 
\underbrace{(-1) + (-1) + \cdots + (-1)}_{\text{$o_3(m)$ times}}
\pmod{4} \\
&= o_1(m) - o_3(m).
\end{align*}
Thus,
\begin{equation*}
\sigma_o \textup{mex}(m)  \equiv o_1(m) - o_3(m) \pmod{4} \quad \text{for all} \ m.
\end{equation*}
Invoking Proposition \ref{HSS} with odd $m (= 2n+1)$ leads us to
\begin{equation*}
\sigma_o \textup{mex}(2n+1) \equiv 0 \pmod 4,
\end{equation*}
as desired. 

\begin{proof}[Proposition {\rm \ref{master proposition}}][]
We rewrite the double series $M(z,q)$, as defined in \eqref{M(z,q)}, as
\begin{align}
\sum_{n=0}^{\infty}\sum_{m=1}^{\infty} x(m,n)z^mq^n=\frac{1}{(q;q)_{\infty}}\sum_{m=1}^{\infty}z^mq^{m\choose2}(1-q^m). \label{M(z,q) double series}
\end{align}
We proceed with the right-hand side of the above identity as follows.
\begin{align}
\frac{1}{(q;q)_{\infty}}\sum_{m=1}^{\infty}z^mq^{m\choose2}(1-q^m)&=\frac{1}{(q;q)_{\infty}}\left(\sum_{m=1}^{\infty}z^mq^{m\choose2}-\sum_{m=1}^{\infty}z^mq^{m+1\choose2}\right)\nonumber
\\&= \sum_{n=0}^{\infty}p(n)q^n   \left(z\sum_{m=0}^{\infty}z^mq^{m+1\choose2}-\sum_{m=1}^{\infty}z^mq^{m+1\choose2}\right)\nonumber
\\
&=\sum_{n=0}^{\infty}p(n)q^n\left(1+(z-1)\sum_{m=0}^{\infty}z^mq^{m+1\choose2}\right)\nonumber
\\&=\sum_{n=0}^{\infty}p(n)q^n+\sum_{n=0}^{\infty}\sum_{m=0}^{\infty}(z-1)z^mp\left(n-\frac{m(m+1)}{2}\right)q^n. \label{Simplified RHS of M(z,q)}
\end{align}
Comparing the coefficient of $q^n$ on the left side of \eqref{M(z,q) double series} with that on the right side of  \eqref{Simplified RHS of M(z,q)}, we obtain \eqref{Master Identity}. 
\end{proof}

\begin{proof}[Theorem \rm{\ref{kth moment of mex}}][] First, recall a useful differential operator, namely, 
$$\mathcal{D}_z:=z \displaystyle\frac{\partial }{\partial z}.$$ Next, observe that $\mathcal{D}_z^k (z^m)=m^kz^m$ for all positive integers $k$.  Hence by applying the operator $k$ times on both sides of \eqref{Master Identity},  we find that
\begin{align*}
\sum_{m=1}^{\infty} x(m,n)\mathcal{D}_z^k (z^m)= \sum_{m=0}^{\infty}p\Big(n- \frac{m(m+1)}{2}\Big)\mathcal{D}_z^k(z^{m+1}-z^m),
\end{align*}
which readily leads to the desired identity \eqref{Diferntial form of Master Identity}. 
\end{proof}
\begin{proof}[Corollary {\rm \ref{application_kth moment}}][]
Simply substitute $k=1$ in \eqref{moment_kth} and \eqref{moment_bar} to obtain \eqref{sigmamex_analogue} and \eqref{sigmabarmex_analogue},  respectively.
\end{proof}
\begin{proof}[Theorem {\rm \ref{kth moment_sigma_odd_even}}][]
Adding the equations \eqref{moment_kth} and \eqref{moment_bar}, we obtain 
\begin{align}
\sigma^{(k)}\text{mex}(n)+{\bar{\sigma}}^{(k)}\text{mex}(n)=\sum_{m=0}^{\infty} & \left(\left\{1+(-1)^m\right\}(m+1)^k-\left\{1-(-1)^m\right\}m^k\right) \nonumber \\
& \times p\left(n- \frac{m(m+1)}{2}\right).\label{sigma_k+ sigmabar_k}
\end{align}
Note that the left-hand side of the above equation is
\begin{align*}
\sigma^{(k)}\text{mex}(n)+{\bar{\sigma}}^{(k)}\text{mex}(n)& =\sum_{\pi\in \mathcal{P}(n)}\left\{ \left(\text{mex}(\pi)\right)^k-(-1)^{\text{mex}(\pi)}\left(\text{mex}(\pi)\right)^k\right\}\\
&=2\sum_{\substack{\pi\in \mathcal{P}(n)\\2\nmid\text{mex}(\pi)}}\left(\text{mex}(\pi)\right)^k \\
&=2 \sigma_o^{(k)}\text{mex}(n).
\end{align*}
Putting this together with \eqref{sigma_k+ sigmabar_k}, we arrive at
\begin{align*}
\sigma_o^{(k)}\text{mex}(n)&=\sum_{m=0}^{\infty}\left(\left\{\frac{1+(-1)^m}{2}\right\}(m+1)^k-\left\{\frac{1-(-1)^m}{2}\right\}m^k\right)p\left(n- \frac{m(m+1)}{2}\right)\label{sigma_k+ sigmabar_k}
\\&=\sum_{m=0}^{\infty}\left\{\delta_m(m+1)^k-(1-\delta_m)m^k\right\}p\left(n- \frac{m(m+1)}{2}\right),
\end{align*}
where $\delta_m = \displaystyle\frac{1+(-1)^m}{2}$. This proves \eqref{General sigma odd formula}. The proof of \eqref{General sigma even formula} is similar except for the minor change that we are required to subtract \eqref{moment_bar} from \eqref{moment_kth} at the beginning.
\end{proof}
\begin{proof}[Corollary {\rm \ref{sigma_odd_even_mex}}][]
It is enough to prove \eqref{sigma odd formula}, the proof of \eqref{sigma even formula} being almost identical. To obtain \eqref{sigma odd formula}, we put $k=1$ in \eqref{General sigma odd formula} to get
\begin{align*}
{\sigma_o}\text{mex}(n)={\sigma_o}^{(1)}\text{mex}(n)=\sum_{m=0}^{\infty}\left\{\delta_m(m+1)-(1-\delta_m)m\right\}p\left(n- \frac{m(m+1)}{2}\right).
\end{align*}
In view of the definitions of $\delta_m$  and $t_m$, we can write the previous equality as
{\allowdisplaybreaks \begin{align*}
\sigma_o\text{mex}(n)&=\sum_{\substack{m=0\\ \text{m even}}}^{\infty}(m+1)p\left(n- t_m\right)-\sum_{\substack{m=1\\ \text{m odd}}}^{\infty}mp\left(n- t_m\right)
\\ &=\sum_{m=0}^{\infty}(2m+1)p\left(n- t_{2m}\right)-\sum_{m=0}^{\infty}(2m+1)p\left(n- t_{2m+1}\right)
\\ &=\sum_{m=0}^{\infty}(2m+1)\left\{p\left(n- t_{2m}\right)-p\left(n- t_{2m+1}\right)\right\}.
\end{align*}}
\end{proof}

\section{An alternate proof of a result due to Hopkins, Sellers and Stanton}\label{sec-five}

In this section, we give an alternative proof of a result due to Hopkins, Sellers and Stanton \cite[Proposition 8]{HoSeStan2022}, which we already stated as Proposition \ref{HSS}. 
\begin{proof-alt} 
Let us recall \eqref{L(z,q)}:
\begin{align*}
\sum_{n=0}^{\infty}\sum_{m=0}^{\infty} x(2m+1,n)z^{2m+1}q^n=\frac{1}{(q;q)_{\infty}}\sum_{m=0}^{\infty}z^{2m+1}q^{{2m+1}\choose2}(1-q^{2m+1}).
\end{align*}
Now we substitute $z=i$ in the above identity to get
\begin{equation}
\sum_{n=0}^{\infty}\sum_{m=0}^{\infty} x(2m+1,n)(-1)^mq^n=\frac{1}{(q;q)_{\infty}}\sum_{m=0}^{\infty}(-1)^mq^{{2m+1}\choose2}(1-q^{2m+1}).\label{ABC}
\end{equation}
First,  we simplify the left-hand side as follows:
\begin{align}
\sum_{n=0}^{\infty}\sum_{m=0}^{\infty} x(2m+1,n)(-1)^mq^n&=\sum_{n=0}^{\infty}\Bigg(\sum_{m=0}^{\infty}  x(4m+1,n)-\sum_{m=0}^{\infty}  x(4m+3,n)\Bigg)q^n\nonumber
\\&=\sum_{n=0}^{\infty}\Big(o_1(n)-o_3(n)\Big)q^n.\label{O1minusO3}
\end{align}
Second, we simplify the right-hand side of \eqref{ABC}:
\allowdisplaybreaks{
\begin{align}
\frac{1}{(q;q)_{\infty}}\sum_{m=0}^{\infty}(-1)^mq^{{2m+1}\choose2}(1-q^{2m+1})
&=\frac{1}{(q;q)_{\infty}}\left(\sum_{m=0}^{\infty}(-1)^mq^{{2m+1}\choose2}-\sum_{m=0}^{\infty}(-1)^mq^{{2m+2}\choose2}\right)\nonumber
\\&=\frac{1}{(q;q)_{\infty}}\left(\sum_{j=0}^{\infty}(-1)^jq^{{2j+1}\choose2}+\sum_{j=-\infty}^{-1}(-1)^jq^{{2j+1}\choose2}\right)\nonumber
\\&=\frac{1}{(q;q)_{\infty}}\sum_{j=-\infty}^{\infty}(-1)^jq^{{2j+1}\choose2}, \label{bilateral series}
\end{align}}
where the rightmost sum in the second equality arises by substituting $j=-(m+1)$.  
Putting \eqref{O1minusO3} and \eqref{bilateral series} into \eqref{ABC} gives us
\begin{align}
\sum_{n=0}^{\infty}\Big(o_1(n)-o_3(n)\Big)q^n=\frac{1}{(q;q)_{\infty}}\sum_{j=-\infty}^{\infty}(-1)^jq^{{2j+1}\choose2}\label{PQR}.
\end{align}
Replacing $a$ by $-q^3$ and $b$ by $-q$ in \eqref{JTPI}, we have
\begin{align}
& \sum_{j=-\infty}^{\infty}(-q^3)^{\frac{j(j+1)}{2}} (-q)^{\frac{j(j-1)}{2}}=(q^3;q^4)_{\infty}(q;q^4)_{\infty}(q^4;q^4)_{\infty};\nonumber
\\\intertext{that is,} & \sum_{j=-\infty}^{\infty}(-1)^j q^{2j^2 + j} = \sum_{j=-\infty}^{\infty}(-1)^j q^{{2j+1}\choose2}=(q;q^4)_{\infty}(q^3;q^4)_{\infty}(q^4;q^4)_{\infty}.\label{closed form for bilateral}
\end{align}
Now utilizing \eqref{closed form for bilateral} in \eqref{PQR}, we arrive at
\begin{align}
\sum_{n=0}^{\infty}\Big(o_1(n)-o_3(n)\Big)q^n & =\frac{1}{(q;q)_{\infty}}\times (q;q^4)_{\infty}(q^3;q^4)_{\infty}(q^4;q^4)_{\infty} \nonumber \\
 &  =\frac{1}{(q^2;q^4)_{\infty}} =(-q^2;q^2)_{\infty},\label{distinct part partition}
\end{align}
where the last step follows by Euler's identity $(-q; q)_{\infty} = \displaystyle \frac{1}{(q; q^2)_{\infty}}
$. Now keeping in mind that $(-q; q)_{\infty}$ is the generating function for partitions of $n$ into distinct parts,  we can rephrase \eqref{distinct part partition} as 
\begin{align*}
& \sum_{n=0}^{\infty}\Big(o_1(n)-o_3(n)\Big)q^n=(-q^2;q^2)_{\infty}=\sum_{n=0}^{\infty}q(n)q^{2n};\\  \intertext{that is,} & \sum_{n=0}^{\infty}\Big(o_1(n)-o_3(n)\Big)q^n=\sum_{n=0}^{\infty}q(n/2)q^n. 
\end{align*}
Since $q(n/2)$ is non-zero only when $n/2$ is a natural number, we thus have
\begin{align*}
o_1(n)-o_3(n)=\begin{cases}
0, & \text{if}~ n~ \text{is odd},
\\q(n/2), & \text{otherwise}.
\end{cases}
\end{align*}
\end{proof-alt}

\section{Concluding Remarks}\label{concluding remarks}
In this paper, we established a refinement of equation \eqref{AndrewsNewman Sigma Identity} of Andrews and Newman on $\sigma\textup{mex}(n)$, namely, Theorem \ref{main refinement}. Moreover, we also proved congruences for $\sigma_o\text{mex}(n)$ and $\sigma_e\text{mex}(n)$ modulo $4$ and $8$, namely, Theorem \ref{congruence}.  Based on computations using Magma and Mathematica, we propose the following conjecture.
\begin{conjecture}\label{main_conjecture}
 Let $n$ be a non-negative integer. Then
\begin{align} \label{conj2}
\sigma_o \textup{mex}(8n+1) &\equiv 0 \pmod {16}\\\intertext{and}
\sigma_e \textup{mex}(8n) & \equiv 0 \pmod {8}. \label{econj2}
\end{align}
\end{conjecture}
Working on the generating function \eqref{gen-4n1} of $\sigma_o \textup{mex}(4n+1)$, we can show that
\begin{align*}
&\sum_{n=0}^\infty\sigma_o \textup{mex}(8n+1)q^n\\
&=\left(f(q^3,q^5)f(q^7,q^9)+q^2f(q,q^7)f(q,q^{15})\right)\left(-\dfrac{f_2^2f_4^5}{f_1^7f_8^2}-\dfrac{f_4^5}{f_1^3f_8^2}+2\dfrac{f_2^7f_4^8}{f_1^{13}f_8^4}+8q\dfrac{f_2^{11}f_8^4}{f_1^{13}f_4^4}\right)\\
&\quad+2q\left(f(q,q^7)f(q^7,q^9)+qf(q^3,q^5)f(q,q^{15})\right)\left(-\dfrac{f_2^4f_8^2}{f_1^7f_4}+\dfrac{f_2^2 f_8^2}{f_1^3f_4}+4\dfrac{f_2^9f_4^2}{f_1^{13}}\right).
\end{align*}
The proposed congruence \eqref{conj2} would follow if one could show that the right side of the above equality vanishes under modulo 16. At this moment, we do not have a proof for that. A  similar generating function expression for $\sigma_e \textup{mex}(8n)$ may also be found.   

Grabner and Knopfmacher \cite[Eq.  (4.7)]{GrabnerKnopfmacher} showed that
$\sigma\textup{mex}(n)$ satisfies the following Hardy-Ramanujan type asymptotic formula:
\begin{align*}
\sigma\textup{mex}(n) \sim  \frac{1}{4  \sqrt[4]{6n^3}} \exp\left( \pi \sqrt{\frac{2n}{3}}   \right) \quad \textrm{as} \, \, n \rightarrow \infty.
\end{align*}
They \cite[Theorem 3]{GrabnerKnopfmacher}  also obtained a Hardy-Ramanujan-Rademacher type exact formula for  $ \sigma\textup{mex}(n)$. It will be highly desirable to obtain asymptotic formulae for $\sigma_o\text{mex}(n)$ and $\sigma_e\text{mex}(n)$ and their moments,  and to do that the identities \eqref{sigma odd formula} and \eqref{sigma even formula} might be useful.  It would also be interesting to find a combinatorial proof of the main result, Theorem \ref{main refinement}, in the manner of Ballantine and Merca \cite{BallantineMerca}, who proved the original result \eqref{AndrewsNewman Sigma Identity} of Andrews and Newman combinatorially.

\textbf{Acknowledgments.}
We thank the anonymous referee for carefully reading our manuscript and giving constructive suggestions. 
The second author wants to thank the Department of Mathematics,  Pt.  Chiranji Lal Sharma Government College,  Karnal for a conducive research environment. The third author is a SERB National Post Doctoral Fellow (NPDF) supported by the fellowship PDF/2021/001090 and would like to thank SERB for the same. The last author wants to thank SERB for the Start-Up Research Grant SRG/2020/000144 and MATRICS Grant MTR/2022/000545.

\end{document}